\documentclass[reqno]{amsart}
\usepackage[english]{babel}
\usepackage{amssymb,graphicx}
\usepackage[T1]{fontenc}

\newtheorem{theorem}{Theorem}[section]

\newtheorem{lemma}[theorem]{Lemma}
\newtheorem{corollary}[theorem]{Corollary}

\theoremstyle{definition}
\newtheorem{definition}[theorem]{Definition}
\newtheorem{example}[theorem]{Example}

\theoremstyle{remark}
\newtheorem*{remark}{Remark}

\numberwithin{equation}{section}

\newcommand{\RE}{\mbox{$\mathbb{R}$}}
\newcommand{\CEP}[1]{\mbox{$\mathbb{C}^{#1}$}}
\newcommand{\C}{\mbox{$\mathbb{C}$}}

\newcommand{\B}{\mbox{$\mathbb{B}$}}

\newcommand{\Eo}{\mbox{$\mathcal{E}_{0}$}}
\newcommand{\Ep}{\mbox{$\mathcal{E}_{p}$}}

\newcommand{\Ec}{\mbox{$\mathcal{E}_{\chi}$}}

\newcommand{\ddcn}[1]{\mbox{$\left(dd^{c}#1\right)^{n}$}}

\newcommand{\vertiii}[1]{{\left\vert\kern-0.25ex\left\vert\kern-0.25ex\left\vert #1\right\vert\kern-0.25ex\right\vert\kern-0.25ex\right\vert}}

\begin{document}

\author{Per \AA hag}
\address{Department of Mathematics and Mathematical Statistics\\ Ume\aa \ University\\ SE-901 87 Ume\aa \\ Sweden}
\email{Per.Ahag@math.umu.se}
\author{Rafa\l\ Czy{\.z}}
\address{Institute of Mathematics\\ Jagiellonian University\\ \L ojasiewicza 6\\ 30-348 Krak\'ow\\ Poland}
\email{Rafal.Czyz@im.uj.edu.pl}
\title{On the Moser--Trudinger inequality in complex space}

\begin{abstract} In this paper we prove the pluricomplex counterpart of the Moser-Trudinger and Sobolev inequalities in complex space. We consider these inequalities for plurisubharmonic functions with finite pluricomplex energy, and estimate the concerned constants.
\end{abstract}

\maketitle

\begin{center}\bf
\today
\end{center}

\section{Introduction}

\bigskip

 Many researchers in partial differential equations and calculus of variation are interested in Sobolev type inequalities, or Sobolev embedding theorems as some wish to call them. The borderline case when the dimension is two is sometimes known as the \emph{Moser-Trudinger inequality} or \emph{Trudinger-Moser inequality} after the work of Trudinger~\cite{trudinger} in 1967 and Moser~\cite{moser} in 1971. To this day these ideas are still used in ongoing research (see e.g.~\cite{beckner,CL,de,LL,PSSW,yuan}). In this paper we shall prove the pluricomplex counterpart to the  Moser-Trudinger and Sobolev inequalities.  We shall now continue with a brief discussion about the setting, and we refer the reader to Section~\ref{sec_Background} for a more detailed background.

 Let $\Omega$ be an open set in $\CEP{n}$. An upper semicontinuous function $u:\Omega\to\RE\cup\{-\infty\}$ is called \emph{plurisubharmonic} if the Laplacian of $u$ is, in the sense of distributions, non-negative along each complex line that intersects $\Omega$. We shall always assume that a plurisubharmonic function is defined on a so called hyperconvex domain $\Omega\subset\CEP{n}$. This assumption is made to ensure a
 satisfying amount of plurisubharmonic functions with certain properties. As the abstract reveal we are interested in plurisubharmonic functions
 with finite pluricomplex energy. To be able to define these functions we start by defining what we recognize as the pluricomplex counterpart of test functions in the theory of distributions. We say that
a plurisubharmonic function $\varphi$ defined on $\Omega$ belongs to $\Eo$ $(=\Eo(\Omega))$ if $\varphi$ is a bounded function,
$\lim_{z\rightarrow\xi} \varphi (z)=0$, for every $\xi\in\partial\Omega$, and $\int_{\Omega} \ddcn{\varphi}<\infty$, where
$\ddcn{\cdot}$ is the complex Monge-Amp\`{e}re operator. Finally, we say that $u\in \Ep$ $(=\Ep(\Omega))$ if $u$ is a plurisubharmonic function
defined on $\Omega$ and there exists a
decreasing sequence, $\{\varphi_{j}\}$, $\varphi_{j}\in\Eo$, that converges pointwise to $u$ on $\Omega$,
as $j$ tends to $\infty$, and
\[
\sup_{j} e_p(\varphi_j)=\sup_{j}\int_{\Omega}
(-\varphi_{j})^{p}\ddcn{\varphi_{j}}< \infty\, .
\]
This definition implies that if $u\in\Ep$, then $e_p(u)<\infty$. This justify that we say that a function $u\in\Ep$ have finite pluricomplex $p$-energy, or simply
finite pluricomplex energy. The energy cones, $\Ep$,  were introduced and studied in~\cite{cegrell_pc}, and the growing use of complex Monge-Amp\`{e}re techniques in applications makes this framework of significant importance (see e.g.~\cite{ACKPZ,BBW,demailly_hiep,Dinh_et al,diller}).

\bigskip

  The first inequality we prove is the following.

\medskip

\noindent \textbf{The pluricomplex Moser-Trudinger inequality.} \emph{Let $\Omega$ be a bounded hyperconvex domain in $\C^n$, $n\geq 2$.  Then there exist constants $A(p,n,\Omega)$ and $B(p,n,\Omega)$ depending only on $p, n, \Omega$ such that for any $u\in \mathcal E_p$ we have that}
\begin{equation}\label{intr_MS}
\log \int_{\Omega}e^{-u}d\lambda_{2n}\leq A(p,n,\Omega)+ B(p,n,\Omega)e_p(u)^{\frac 1p}\, ,
\end{equation}
\emph{where $d\lambda_{2n}$ is the Lebesgue measure in $\mathbb C^n$. For any given $0<\epsilon<1$ we can take the constants $A(p,n,\Omega)$ and $B(p,n,\Omega)$ to be}
\begin{align*}
A(p,n,\Omega)& =\log \left(\left(\pi^n+\beta(n)\frac {\epsilon n}{(n-\epsilon n)^n}\right)\operatorname{diam}(\Omega)^{2n}\right) \, \text { and } \\[2mm] B(p,n,\Omega) & =(2\epsilon n)^{-\frac np}\, ,
\end{align*}
\emph{where $\beta(n)$ is a constant depending only on $n$. Furthermore, we have the following estimates on $B(p,n,\Omega)$}
\begin{equation}\label{1.2}
\frac {p}{(4\pi)^{\frac np}(n+p)^{1+\frac np}}\leq B(p,n,\Omega)\, ,
\end{equation}
\emph{and if $\Omega=\B$ is the unit ball, then we have that}
\begin{equation}\label{1.3}
\frac {p}{(4\pi)^{\frac np}(n+p)^{1+\frac np}}\leq B(p,n,\mathbb{B})\leq \left(\frac {p^{p-1}}{(4\pi)^n(n+1)^{n+1}(n+p)^{p-1}}\right)^{\frac 1p}\, .
\end{equation}

\medskip

The proof of this theorem is divided into parts. Inequality~(\ref{intr_MS}) is proved in Theorem~\ref{t11}. This inequality was first proved for $p=1$ in~\cite{C1},
and two months later another proof appeared~\cite{BB1} that generated slightly better estimates. In Theorem~\ref{Best} we present the proof of our estimates. In the proof of the upper bound we are using a slightly modified version of Moser's original inequality.

\bigskip

In Section~\ref{sec_sobolev} we shall prove with help of the pluricomplex Moser-Trudinger inequality the following inequality.

\medskip

\noindent \textbf{The pluricomplex Sobolev inequality.} \emph{Let $\Omega$ be a bounded hyperconvex domain in $\C^n$, $n\geq 2$, and let $u\in \Ep$, $p>0$. Then for all $q>0$ there exists a constant $C(p,q,n,\Omega)>0$ depending only on $p,q, n,\Omega$ such that}
\begin{equation}\label{intr_lq theorem}
\|u\|_{L^q}\leq C(p,q,n,\Omega)e_p(u)^{\frac {1}{n+p}}\, .
\end{equation}
\emph{In fact one can take}
\begin{equation}\label{intr_const c}
C(p,q,n,\Omega)=e^{\frac 1qA(p,n,\Omega)}\frac {(n+p)B(p,n,\Omega)^{\frac {p}{n+p}}}{n^{\frac {n}{n+p}}p^{\frac {p}{n+p}}}\, \Gamma\left(\frac {nq}{n+p}+1\right)^{\frac 1q},
\end{equation}
\emph{where the constants $A(p,n,\Omega)$ and $B(p,n,\Omega)$ are given in the above Moser-Trudinger inequality. Furthmore, $\Gamma$ denotes the gamma function.   In addition, inequality (\ref{intr_lq theorem}) may be written in the form}
\[
\|u\|_{L^q}\leq D(p,n,\Omega)q^{\frac {n}{n+p}}e_p(u)^{\frac {1}{n+p}} \, ,
\]
\emph{where the constant $D(p,n,\Omega)$ does not depend on $q$.}

\medskip

The pluricomplex Sobolev inequality is proved in Theorem~\ref{lq estimate}. It should be noted that Theorem~\ref{lq estimate} is more general than the above statement, this due to presentational reasons. Our work in Section~\ref{sec_sobolev} was inspired by~\cite{BB1,Blo96,CP2,kol}. Next let $C(p,q,n,\mathbb{B})$ denote the optimal constant in~(\ref{intr_lq theorem}), i.e. the infimum of all admissible constants. This optimal constant is classically of great importance. For example it is connected
to the isoperimetric inequality and therefore classically to symmetrization of functions (see e.g.~\cite{talenti}). In pluripotential theory there have been many attempts to symmetrize plurisubharmonic functions, but few progress have been made in that direction since plurisubharmonicity might be lost during a symmetrization procedure (\cite{Baernstein}). For a positive result see~\cite{BB2}. A strong trend today is to try to prove a pluricomplex counterpart of Talenti's theorem for the Laplacean  (\cite{talenti2}, see also \cite[Theorem 10.2]{Blo14}). A successful attempt would not only imply simplified proofs, but also many of the biggest unsolved problems would be conquered. For further information and details we refer to Section~10 in the excellent survey~\cite{Blo14} written by B\l ocki. With this in mind we shall in Section~\ref{sec_optimal} prove that
\[
C(1,1,n,\mathbb{B})=\frac {\pi^{\frac {n^2}{n+1}}}{4^{\frac {n}{n+1}}n!(n+1)^{\frac {n}{n+1}}}\, .
\]

\section{Background}\label{sec_Background}

 In this section we shall give some necessary background on pluripotential theory. For further information we refer to~\cite{cegrell_gdm,hab,demailly_bok,klimek,kol2,PSS}.

 A set $\Omega\subseteq\CEP{n}$, $n\geq 1$, is called a \emph{bounded hyperconvex domain} if it is a bounded, connected, and open set such
 that there exists a bounded plurisubharmonic function $\varphi:\Omega\rightarrow (-\infty,0)$ such that the closure of the
 set $\{z\in\Omega : \varphi(z)<c\}$ is compact in $\Omega$, for every $c\in (-\infty, 0)$.

We say that a plurisubharmonic function $\varphi$ defined on $\Omega$ belongs to $\Eo$ $(=\Eo
(\Omega))$ if it is a bounded function,
\[
\lim_{z\rightarrow\xi} \varphi (z)=0 \quad \text{ for every } \xi\in\partial\Omega\, ,
\]
and
\[
\int_{\Omega} \ddcn{\varphi}<\infty\, ,
\]
where $\ddcn{\cdot}$ is the complex Monge-Amp\`{e}re operator. Furthermore, we say that $u\in \Ep$ $(=\Ep(\Omega))$, $p>0$, if $u$ is a plurisubharmonic function
defined on $\Omega$ and there exists a decreasing sequence, $\{\varphi_{j}\}$, $\varphi_{j}\in\Eo$, that converges pointwise to $u$ on $\Omega$,
as $j$ tends to $\infty$, and
\[
\sup_{j} e_p(\varphi_j)=\sup_{j}\int_{\Omega}
(-\varphi_{j})^{p}\ddcn{\varphi_{j}}< \infty\, .
\]
We shall need on several occasions the following two inequalities. The inequality in Lemma~\ref{est} follows by standard approximation techniques from the
work of B\l ocki in~\cite{Blo93}.
\begin{lemma}\label{est} Let $u\in \Eo$, $v\in\Ep$, $p>0$, $1\leq k\leq n$. Then
\[
\int_{\Omega} (-v)^{p+k}(dd^c u)^n\leq (p+k)\cdots (p+1)\|u\|^k_{L^{\infty}}\int_{\Omega}(-v)^{p}(dd^cv)^k\wedge(dd^cu)^{n-k}.
\]
\end{lemma}

Next theorem  was proved in~\cite{persson} for $p\geq 1$, and for $0<p<1$
in~\cite{czyz_energy} (see also~\cite{cegrell_pc,CP}).

\begin{theorem}\label{thm_holder} Let $p>0$ and $u_0,u_1,\ldots ,
u_n\in\Ep$. If $n\geq 2$, then
\begin{multline*}
\int_\Omega (-u_0)^p dd^c u_1\wedge\cdots\wedge dd^c u_n\\ \leq
d(p,n,\Omega)\; e_p(u_0)^{p/(p+n)}e_p(u_1)^{1/(p+n)}\cdots
e_p(u_n)^{1/(p+n)}\, ,
\end{multline*}
where
\[d(p,n,\Omega)=
\begin{cases}
 \left(\frac1p\right)^{\frac {n}{n-p}} & \text{if $0<p<1$,}\\

1 & \text{if $p=1$,}\\
p^{\frac{p\alpha (n,p)}{p-1}} & \text{if $p>1$,}
\end{cases}
\]
and $\alpha (n,p)=(p+2)\left(\frac{p+1}{p}\right)^{n-1}-(p+1)$.
\end{theorem}

\section{The pluricomplex Moser-Trudinger inequality}\label{sec_Moser-Trudinger}

The aim of this section is to prove the following Moser-Trudinger inequality for $\Ep$, $p>0$.

\begin{theorem}\label{t11}
Let $\Omega$ be a bounded hyperconvex domain in $\mathbb C^n$. Then there exist constants $A(p,n,\Omega)$ and $B(p,n,\Omega)$ depending only on $p, n, \Omega$ such that for any $u\in \mathcal E_p$ we have that
\[
\log \int_{\Omega}e^{-u}d\lambda_{2n}\leq A(p,n,\Omega)+ B(p,n,\Omega)e_p(u)^{\frac 1p}\, .
\]
Furthermore, for any given $0<\epsilon<1$ we can take the constants $A(p,n,\Omega)$ and $B(p,n,\Omega)$ to be
\[
A(p,n,\Omega)=\log \left(\left(\pi^n+\beta(n)\frac {\epsilon n}{(n-\epsilon n)^n}\right)\operatorname{diam}(\Omega)^{2n}\right) \, \text { and } \,  B(p,n,\Omega)=(2\epsilon n)^{-\frac np}\, ,
\]
where $\beta(n)$ is a constant depending only on $n$.
\end{theorem}
\begin{proof} First assume that $u\in \Eo(\Omega)\cap C(\bar{\Omega})$. Next, thanks to~\cite{cegrell_gdm} we can find a uniquely determined function $w\in \Eo$ that
satisfies
\[
(dd^cw)^n=(-u)^p(dd^cu)^n\, .
\]
We shall now prove that
\begin{equation}\label{t11_1}
u\geq t^{-\frac pn}w-t\qquad \text{for all } t>0\, .
\end{equation}
First notice that on the set $\{z\in \Omega: u(z)\geq -t\}$ we have that
\[
u(z)\geq -t\geq t^{-\frac pn}w-t\, .
\]
Next, since  $u\in \Eo$ and $\lim_{z\to \partial \Omega}u(z)=0$ we have that the open set $\omega=\{z\in \Omega: u(z)< -t\}$ is relatively compact in $\Omega$ and
therefore
\[
t^p(dd^cu)^n\leq (-u)^p(dd^cu)^n=(dd^cw)^n\, .
\]
Hence,
\[
(dd^cu)^n\leq t^{-p}(dd^cw)^n=\Big(dd^c(t^{-\frac pn}w-t)\Big)^n\, ,
\]
and furthermore
\[
\liminf_{\omega\ni z\to \partial \omega}(u(z)-t^{-\frac pn}w(z)+t)\geq 0\, .
\]
Therefore, by the comparison principle (see~\cite{cegrell_pc}) we get that $u\geq t^{-\frac pn}w-t$ on $\omega$ and~(\ref{t11_1}) is valid.

Fix $0<\epsilon<1$ and choose $t$ such that
\[
t=\left(\frac {e_p(u)}{\epsilon^n(2n)^n}\right)^{\frac 1p}\, .
\]
With this choice of $t$ we have
\[
\int_{\Omega}(dd^ct^{-\frac pn}w)^n=t^{-p}\int_{\Omega}(-u)^p(dd^cu)^n=t^{-p}e_p(u)=\epsilon^n (2n)^n\, .
\]
By using Corollary~5.2 in~\cite{ACKPZ} for the function $t^{-\frac pn}w$ we get
\begin{multline*}
\log \int_{\Omega}e^{-u}d\lambda_{2n}\leq \log \int_{\Omega}e^{-t^{-\frac pn}w}e^{t}d\lambda_{2n}\leq \\
\leq \log \left(\left(\pi^n+\beta(n)\frac {\epsilon n}{(n-\epsilon n)^n}\right)\operatorname{diam}(\Omega)^{2n}\right)+ \frac {e_p(u)^{\frac 1p}}{(2\epsilon n)^{\frac np}}\, .
\end{multline*}

By a standard procedure we can now remove the assumption that $u\in \Eo(\Omega)\cap C(\bar{\Omega})$, since for arbitrary $u\in \Ep$ there exists a sequence $u_j\in\Eo(\Omega)\cap C(\bar{\Omega})$ such that $u_j\searrow u$ and $e_p(u_j)\to e_p(u)$, $j\to \infty$ (see e.g.~\cite{CKZ}).
\end{proof}

Next in Corollary~\ref{t11_cor} we obtain Theorem~\ref{t11} for functions from the class $\mathcal E_{\chi}$. Let us first recall the definition of $\mathcal E_{\chi}$ (see e.g.~\cite{bgz,Guedj_Zer} for further information).

\begin{definition}\label{def_Echi} Let $\Omega$ be a bounded hyperconvex domain in $\mathbb C^n$, $n\geq 2$. Let $\chi:(-\infty,0]\to (-\infty,0]$ be a continuous and
nondecreasing function. Furthermore, let $\Ec$ contain those plurisubharmonic functions
$u$ for which there exists a decreasing sequence $u_{j}\in\Eo$ that converges pointwise to $u$ on
$\Omega$, as $j$ tends to $\infty$, and
\[
e_{\chi}(u)=\sup_{j}\int_{\Omega} -\chi(u_j)\ddcn{u_j}< \infty\, .
\]
\end{definition}

For example, with this notation if $\chi=-(-t)^p$, then $\Ec=\Ep$. It was proved in~\cite{b} and in~\cite{haihiep} that if $\chi:(-\infty,0]\to (-\infty,0]$ is
continuous, and  strictly increasing, then the complex Monge-Amp\`{e}re operator is well defined on $\Ec$. We are now in position to prove  Corollary~\ref{t11_cor}.

\begin{corollary}\label{t11_cor} Let $\chi:(-\infty,0]\to (-\infty,0]$ be a continuous and nondecreasing function, and let $\Omega$ be a bounded hyperconvex domain in $\mathbb C^n$, $n\geq 2$. Then for any fixed  $0<\epsilon<1$ we have for all $u\in \Ec$ that
\[
\log \int_{\Omega}e^{-u}d\lambda_{2n}\leq \log \left(\left(\pi^n+\beta(n)\frac {\epsilon n}{(n-\epsilon n)^n}\right)\operatorname{diam}(\Omega)^{2n}\right)- \chi^{-1}\left(\frac {-e_{\chi}(u)}{(2\epsilon n)^n}\right).
\]
\end{corollary}
\begin{proof} The proof of Theorem~\ref{t11} works here as well with some changes of references. The reference~\cite{cegrell_gdm} should be replaced with~\cite{haihiep}, and~\cite{CKZ} should be replaced with~\cite{b}.
\end{proof}

We shall end this section with a remark about the case when the underlying space is a compact K\"ahler manifold. Let us first recall some facts. Let $(X,\omega)$ be a K\"ahler manifold of dimension $n$ with a K\"ahler form $\omega$ such that $\int_{X}\omega^n=1$. We say that $u\in \Ep(X,\omega) (=\Ep)$ if there exists a sequence $u_j\in \mathcal {PSH}(X,\omega)\cap L^{\infty}(X)$ such that $u_j\leq 0$, $u_j\searrow u$, $j\to \infty$ and
\[
\sup_{j}\int_X(-u_j)^p(dd^cu_j+\omega)^n<\infty\, .
\]
Here $\mathcal {PSH}(X,\omega)$ denote the set of $\omega$-plurisubharmonic functions. For $u\in \Ep$, set
\[
e_p(u)=\int_X(-u)^p(dd^cu+\omega)^n\, .
\]
In the case $p=1$ , we have the following classical functional defined on $\mathcal E_1$ by
\[
\mathcal E_{\omega}(u)=\frac {1}{(n+1)!}\sum_{k=0}^{n}\int_X(-u)(dd^cu+\omega)^k\wedge \omega^{n-k}\, ,
\]
and we have the following estimation
\[
\mathcal E_{\omega}(u)\leq \frac {1}{n!}\int_X(-u)(dd^cu+\omega)^n=\frac {1}{n!}\, e_1(u)\, .
\]

\begin{remark}\label{kahler} Let $(X,\omega)$ be a compact K\"ahler manifold of dimension $n$ with a K\"ahler form $\omega$ such that $\int_{X}\omega^n=1$. It was proved
in~\cite{BB1} that there exist constants $a,b>0$ such that for all $u\in \mathcal E_1$ and $k>0$ it holds that
\begin{equation}\label{kahler_1}
\log \left (\int_Xe^{-ku}\omega^n\right)\leq ak^{n+1}\mathcal E_{\omega}(u)+b\, .
\end{equation}
Now let $u\in \Ep$, $p>1$. By using H\"older inequality we get that
\[
e_1(u)=\int_X(-u)(dd^cu+\omega)^n\leq \left(\int_X(-u)^p(dd^cu+\omega)^n\right)^{\frac 1p}=e_p(u)^{\frac 1p}\, .
\]
Thus,  $u\in \mathcal E_p$,  and by~(\ref{kahler_1}) we arrive at
\begin{multline}\label{kahler 2}
\log \left (\int_Xe^{-ku}\omega^n\right)\leq ak^{n+1}\mathcal E_{\omega}(u)+b\leq \frac {ak^{n+1}}{n!}e_1(u)+b\\ \leq \frac {ak^{n+1}}{n!}e_p(u)^{\frac 1p}+b .
\end{multline}
This inequality shall be used on page~\pageref{kahler2}. It should be noted that the case when $0<p<1$ is at this point unknown to the authors.
\end{remark}

\section{Estimates of the constant $B(p,n,\Omega)$}\label{sec_Estimates}

Let us introduce the following notation. For $r>0$ let $\mathbb{B} (z_0,r)=\{z\in\CEP{n}: |z-z_0|< r\}$ be the open ball with center $z_0$ and radius $r$, and to
simplify the notations set  $\mathbb{B}=\mathbb{B} (0,1)$.

Now let $B(p,n,\Omega)$ denotes the optimal constant in the Moser-Trudinger inequality (\ref{intr_MS}), i.e. the infimum of all admissible constants. The aim of this section is to estimate the constant $B(p,n,\Omega)$ for arbitrary hyperconvex domains, inequality (\ref{1.2}), and also in the special case when $\Omega=\B$, inequality (\ref{1.3}).
We shall arrive to the following estimates.

\begin{theorem}\label{Best} Let $\Omega$ be a bounded hyperconvex domain in $\mathbb C^n$, and $B(p,n,\Omega)$ the constant in Theorem~\ref{t11}. Then we have that
\[
\frac {p}{(4\pi)^{\frac np}(n+p)^{1+\frac np}}\leq B(p,n,\Omega)\, .
\]

\end{theorem}
\begin{proof}
Without loss of generality we can assume that $0\in \Omega$. Let $g_{\Omega}(z,0)$ be the pluricomplex Green function with pole at $0$, and for a parameter $\beta\leq 0$ let us define
\[
u(z)=(2n+2p)\max\big(g_{\Omega}(z,0),\beta\big)\, .
\]
This construction yields that
\[
e_p(u)=\int_{\Omega}(-u)^p(dd^cu)^n=(2\pi)^n(2n+2p)^{p+n}(-\beta)^p\, ,
\]
and then we shall proceed by estimating the integral
\begin{equation}\label{eq1}
\int_{\Omega}e^{-u}d\lambda_{2n}\, .
\end{equation}
From the definition of the pluricomplex Green function it follows that there exist a radius $r>0$, and a constant $C>0$, such that $\B(0,r)\Subset\Omega$ and such that
the following inequalities hold for all $z\in \B(0,r)$:
\begin{equation}\label{eq2}
\log|z|-C\leq g_{\Omega}(z,0)\leq \log |z|+C\, .
\end{equation}
Choose then $\beta\leq \beta_1\leq 0$ such that it holds $\{z\in \Omega:g_{\Omega}(z,0)<\beta_1\}\subset \B(0,r)$. From now on we shall only consider those $\beta$ with
$\beta\leq \beta_1$. From~(\ref{eq2}) we now have that
\[
\B\left(0,e^{\beta-C}\right)\subset \{z\in \Omega:g_{\Omega}(z,0)<\beta\}\subset \B\left(0,e^{\beta+C}\right)\, .
\]
We start by dividing~(\ref{eq1}) as
\begin{multline*}
\int_{\Omega}e^{-u}d\lambda_{2n}=\int_{\{z\in \Omega:g_{\Omega}(z,0)<\beta\}}e^{-(2n+2p)\beta}d\lambda_{2n}\\+\int_{\{z\in \Omega:g_{\Omega}(z,0)\geq \beta\}}e^{-(2n+2p) g_{\Omega}(z,0)}d\lambda_{2n}=I_1+I_2\ ,
\end{multline*}
and notice that by~(\ref{eq2}) we have
\begin{multline}\label{eq3}
\frac {\pi^n}{n!}e^{-2nC}e^{-2p\beta}=e^{-(2n+2p)\beta}\lambda_{2n}(\B(0,e^{\beta-C}))\leq I_1\\ \leq e^{-(2n+2p)\beta}\lambda_{2n}(\B(0,e^{\beta+C})=\frac {\pi^n}{n!}e^{2nC}e^{-2p\beta}\, .
\end{multline}
Furthermore,
\begin{multline*}
I_2=\int_{\Omega\setminus \mathbb{B}(0,r)}e^{-(2n+2p)g_{\Omega}(z,0)}d\lambda_{2n}+\int_{\mathbb{B}(0,r)\cap \{z\in \Omega:g_{\Omega}(z,0)\geq \beta\}}e^{-(2n+2p)g_{\Omega}(z,0)}d\lambda_{2n}\\=I_3+I_4\, .
\end{multline*}
For $z\in \Omega\setminus \B(0,r)$ we have that
\[
1\leq e^{-(2n+2p)g_{\Omega}(z,0)}\leq e^{-(2n+2p)\beta_1}
\]
and therefore
\begin{equation}\label{eq4}
\lambda_{2n}(\Omega\setminus \B(0,r))\leq I_3\leq e^{-(2n+2p)\beta_1}\lambda_{2n}(\Omega\setminus \B(0,r))\, .
\end{equation}
We also get the estimate of $I_4$ as
\begin{multline}\label{eq5}
0\leq I_4\leq \int_{\mathbb{B}(0,r)\setminus \mathbb{B}(0,e^{\beta-C})}e^{-(2n+2p)(\log|z|-C)}d\lambda_{2n}\\=e^{(2n+2p)C}\frac {2\pi^n}{(n-1)!}\int_{e^{\beta-C}}^r t^{-1-2p}dt
=\frac {2\pi^ne^{(2n+2p)C}}{(n-1)!(-2p)}\left(r^{-2p}-e^{(\beta-C)(-2p)}\right)\, .
\end{multline}
From (\ref{eq1}), (\ref{eq3}), (\ref{eq4}) and (\ref{eq5}) it follows that there exist constant $c_1, c_2, c_3, c_4$ not depending on $\beta$ such that
\[
c_1e^{-2p\beta}+c_2\leq \int_{\Omega}e^{-u}d\lambda_{2n}\leq c_3e^{-2p\beta}+c_4\, ,
\]
and therefore
\[
\lim_{\beta\to -\infty}\,\frac {\log \left(\displaystyle{\int_{\Omega}e^{-u}d\lambda_{2n}}\right)}{e_p(u)^{\frac 1p}}=\frac {p}{(4\pi)^{\frac np}(n+p)^{1+\frac np}}\, .
\]
Thus,
\[
B(p,n,\Omega)\geq \frac {p}{(4\pi)^{\frac np}(n+p)^{1+\frac np}}\, .
\]
\end{proof}

To prove the inequality (\ref{1.3}) we shall make use of radially symmetric plurisubharmonic functions. Let us recall some basic facts here, and we refer the reader
to~\cite{AC1,monn} and the references therein for further information. Recall that a function $u:\mathbb{B}\to [-\infty,\infty)$ is said to be \emph{radially symmetric}
if we have that
\[
u(z)=u(|z|)\qquad \text{ for all } z\in\mathbb{B}\, .
\]
For each radially symmetric function $u:\mathbb{B}\to [-\infty, \infty)$ we define the function $\tilde u:[0,1)\to [-\infty, \infty)$ by
\begin{equation}\label{prel_def_rad}
\tilde u(t)=u(|z|)\, , \text{ where } t=|z|\, .
\end{equation}
On the other hand, to every function $\tilde v:[0,1)\to [-\infty, \infty)$ we can construct a radially symmetric function $v$
through~(\ref{prel_def_rad}). Furthermore, $u$ is a radially symmetric plurisubharmonic function if and only if $u(t)$ is an increasing function, and it is convex with respect to $\log t$.

\bigskip

Let us first show a few elementary lemmas.

\begin{lemma}\label{l1}
For any $\alpha >1$, and any $A>0$, there exists a constant $B$ such that for all $t\geq 0$ it holds
\[
At^{\alpha}+B\geq t .
\]
In fact one can take
\[
B=\frac {\alpha-1}{\alpha}(\alpha A)^{\frac {1}{1-\alpha}}\, .
\]
\end{lemma}
\begin{proof}
It is is enough to observe that the function
\[
f(t)=At^{\alpha}-t
\]
attains its minimum at
\[
t_0=\left(\alpha A\right)^{\frac {1}{1-\alpha}}\, ,
\]
and that
\[
\min_{[0,\infty)} f=\frac {1-\alpha}{\alpha}(\alpha A)^{\frac {1}{1-\alpha}}=-B\, .
\]
\end{proof}

In Lemma~\ref{ep lemma} we shall make use of the following equality. For $f\in L^{p}(X,\mu)$ we have
\begin{equation}\label{ep eq_1}
\int_X|f|^p\,d\mu=p\int_{0}^{\infty}t^{p-1}\mu(\{x\in X: |f(x)|\geq t\})\,dt\, .
\end{equation}

\begin{lemma}\label{ep lemma}
Let $p>0$, and let $u(z)=u(|z|)=\tilde u(t)$ be a radially symmetric plurisubharmonic function such that $\lim_{z\to \partial \mathbb{B}}u(z)=0$ and $u\in \mathcal E_p$, then we have
\begin{equation}\label{ep eq}
e_p(u)=(2\pi)^np\int_0^1(-\tilde u(t))^{p-1}\tilde u'(t)^{n+1}t^ndt\, .
\end{equation}
\end{lemma}
\begin{proof} If $u(z)=u(|z|)=\tilde u(t)$ is a radially symmetric plurisubharmonic function such that $\lim_{z\to \partial \mathbb{B}}u(z)=0$, then for $t=|z|$ it holds
that
\[
F(t):=\frac {1}{(2\pi)^n}(dd^cu)^n(B(0,t))=t^n\tilde u'(t)^n\, ,
\]
where $\tilde u'$ is the left derivative of a convex function $\tilde u$ (see \cite{AC1}). For $t\geq 0$ we have that
\[
\{z\in \mathbb{B}:u(z)\leq -t\}=B(0,s), \, \text { where } \, s=\tilde u^{-1}(-t),
\]
where $\tilde u^{-1}(\inf u)=\sup\{x: \tilde u(x)=\inf \tilde u\}$. Therefore, by using~(\ref{ep eq_1}) we arrive at
\begin{multline*}
e_p(u)=\int_{\mathbb{B}}(-u)^p(dd^cu)^n=p\int_0^{-\inf u}t^{p-1}(dd^cu)^n(\{z\in \mathbb{B}: u(z)\leq -t\})\, dt \\
=p(2\pi)^n\int_0^{-\inf \tilde u}t^{p-1}F(\tilde u^{-1}(-t))\, dt
=(2\pi)^np\int_0^1(-\tilde u(s))^{p-1}\tilde u'(s)^{n+1}s^n\,ds\, ,
\end{multline*}
where $\tilde u(s)=t$, and this completes this proof.
\end{proof}

We are now in position to prove the inequality (\ref{1.3}).

\begin{theorem}\label{rsep}
Let $p>0$, and let $u$ be a radially symmetric plurisubharmonic function such that $\lim_{z\to \partial \mathbb{B}}u(z)=0$ and $u\in \mathcal E_p$, then we have that

\[
\log \int_{\mathbb{B}}e^{-u(z)}d\lambda_{2n}\leq d+ \left(\frac {e_p(u)p^{p-1}}{(4\pi)^n(n+1)^{n+1}(n+p)^{p-1}}\right)^{\frac 1p}\, ,
\]
where the constant $d$ does not depend on $u$. Therefore,
\[
B(p,n,\mathbb{B})\leq \left(\frac {p^{p-1}}{(4\pi)^n(n+1)^{n+1}(n+p)^{p-1}}\right)^{\frac 1p}\, .
\]
\end{theorem}
\begin{proof}
By Lemma~\ref{ep lemma} we have that the pluricomplex $p$-energy of $u$ is equal to
\begin{multline*}
e_p(u)=(2\pi)^np\int_0^1(-\tilde u(t))^{p-1}\tilde u'(t)^{n+1}t^ndt=\\
=\frac {(2\pi)^np(n+1)^{n+1}}{(n+p)^{n+1}}\int_0^1\left(\left(-(-\tilde u(t))^{\frac {n+p}{n+1}}\right)'\right)^{n+1}t^ndt\, .
\end{multline*}
Therefore, if $v(t)=-(-x\tilde u(t))^{\frac {n+p}{n+1}}$, where
\[
x=\left(\frac {(2\pi)^np(n+1)^{n+1}}{e_p(u)(n+p)^{n+1}}\right)^{\frac {1}{n+p}}\, ,
\]
then $v$ be an increasing function $v:[0,1)\to (-\infty,0)$ such that $\lim_{t\to 1^-}v(t)=0$ and
\[
\int_{0}^1(v'(s))^{n+1}s^nds\leq 1.
\]
Thanks to a slightly modified version of the classical Moser-Trudinger inequality (cf.~\cite{moser}), we arrive at
\begin{equation}\label{moser3}
\int_{0}^1e^{2n(-v(s))^{\frac {n+1}{n}}}s^{2n -1}ds=\int_{0}^1e^{2n(-\tilde u(s)x)^{\frac {n+p}{n}}}s^{2n -1}ds\leq \frac {c}{2n}\, .
\end{equation}
where the constant $c$ does not depend on $u$. Lemma~\ref{l1} yields that
\[
-\tilde u(s)\leq 2n(-\tilde u(s)x)^{\frac {p+n}{n}}+\left(\frac {e_pp^{p-1}}{(4\pi)^n(n+1)^{n+1}(n+p)^{p-1}}\right)^{\frac 1p}= 2n(-\tilde u(s)x)^{\frac {p+n}{n}}+y\, .
\]
Hence by (\ref{moser3}),
\begin{multline*}
\int_{\mathbb{B}}e^{-u(z)}d\lambda_{2n}=\frac {2\pi^n}{(n-1)!}\int_0^1e^{-\tilde u(s)}s^{2n-1}ds\leq \\
\frac {2\pi^n}{(n-1)!}\int_0^1 e^{2n(-\tilde u(s)x)^{\frac {n+p}{n}}+y}s^{2n -1}ds\leq \frac {2\pi^n}{(n-1)!}\frac {c}{2n}e^y\, ,
\end{multline*}
and finally
\[
\log \int_{\mathbb{B}}e^{-u(z)}d\lambda_{2n}\leq \log \left(\frac {\pi^n c}{n!}\right)+ \left(\frac {e_pp^{p-1}}{(4\pi)^n(n+1)^{n+1}(n+p)^{p-1}}\right)^{\frac 1p}\, .
\]
\end{proof}

A direct consequence of~(\ref{moser3}) is the following corollary which was first proved in~\cite{BB2} in the case $p=1$.

\begin{corollary}
Let $p>0$, and let $u$ be a radially symmetric plurisubharmonic function such that $\lim_{z\to \partial \mathbb{B}}u(z)=0$ and $u\in \mathcal E_p(\mathbb{B})$, then we have that
\[
\int_{\mathbb{B}}e^{\alpha(p,n)(-u(z))^{\frac {n+p}{n}}e_p(u)^{-\frac 1n}}d\lambda_{2n}<\infty\, ,
\]
where $\alpha(p,n)=4\pi np^{\frac 1n}\left(\frac {n+1}{n+p}\right)^{\frac {n+1}{n}}$.
\end{corollary}
\begin{proof}
By (\ref{moser3}) we have
\begin{multline*}
\int_{\mathbb{B}}e^{\alpha(p,n)(-u(z))^{\frac {n+p}{n}}e_p(u)^{-\frac 1n}}d\lambda_{2n}=\frac {2\pi^n}{(n-1)!}\int_0^1e^{\alpha(p,n)(-\tilde u(s))^{\frac {n+p}{n}}e_p(u)^{-\frac 1n}}s^{2n-1}\, ds \\
 < \frac {c\pi^n}{n!}<\infty.
\end{multline*}
\end{proof}

\section{The pluricomplex Sobolev inequality}\label{sec_sobolev}

In this section we shall prove the pluricomplex Sobolev inequality. We shall prove it for differences of plurisubharmonic functions with finite energy, i.e.
 for functions in $\delta\Ep=\Ep-\Ep$. If we for $u=u_1-u_2\in \delta\Ep$ define $\vertiii{u}_p$ by
\[
\vertiii{u}_p=\inf_{u_1-u_2=u \atop u_1,u_2\in \mathcal{E}_p} e_{p}(u_1+u_2)^{\frac {1}{n+p}}\, ,
\]
then $(\delta\Ep, \|\cdot\|_p)$ becomes a quasi-Banach space, and for $p=1$ a Banach space (see \cite{mod}). Note that in the case $u\in \Ep$ we have that  $\vertiii{u}_p=e_p(u)^{\frac 1{n+p}}$.

\begin{theorem}\label{lq estimate}
Let $\Omega$ be a bounded hyperconvex domain in $\C^n$, and let $u\in \delta\Ep$, $p>0$. Then for all $q>0$ there exists a
constant $C(p,q,n,\Omega)>0$ depending only on $p,q, n,\Omega$ such that
\begin{equation}\label{lq theorem}
\|u\|_{L^q}\leq C(p,q,n,\Omega)\vertiii{u}_p\, .
\end{equation}
In fact one can take
\begin{equation}\label{const c}
C(p,q,n,\Omega)=e^{\frac 1qA(p,n,\Omega)}\frac {(n+p)B(p,n,\Omega)^{\frac {p}{n+p}}}{n^{\frac {n}{n+p}}p^{\frac {p}{n+p}}}\, \Gamma\left(\frac {nq}{n+p}+1\right)^{\frac 1q},
\end{equation}
where the constants $A(p,n,\Omega)$ and $B(p,n,\Omega)$ are given in Theorem~\ref{t11}.
In addition, inequality (\ref{lq theorem}) may be written, for $q\geq 1$, in the form
\begin{equation}\label{q}
\|u\|_{L^q}\leq D(p,n,\Omega)q^{\frac {n}{n+p}}\vertiii{u}_p\, ,
\end{equation}
where the constant $D(p,n,\Omega)$ does not depend on $q$.
Furthermore, the identity operator $\iota:\delta\Ep\to L^q$ is compact.
\end{theorem}

Before we start the proof let us recall the definition of compactness in quasi-Banach spaces.

\begin{definition}
Let $X$, $Y$ be two quasi-Banach spaces. The operator $K:X\to Y$ is called \emph{compact} if for any sequence $\{x_n\}\subset X$ with $\|x_n\|\leq 1$, then there exists a convergent subsequence $\{y_{n_k}\}$  of $\{K(x_n)\}$.
\end{definition}

\begin{proof}[Proof of Theorem~\ref{lq estimate}]
First assume that $u\in \Ep$, $p>0$. For $t,s>0$ define
\[
f(t)=\int_{\Omega}e^{-tu}d\lambda_{2n} \, \text { and } \, \lambda(s)=\lambda_{2n}(\{z\in\Omega: u(z)<-s\}).
\]
Note that by Theorem~\ref{t11} there exist constants $A=A(p,n,\Omega)$ and $B=B(p,n,\Omega)$ such that
\[
f(t)\leq e^Ae^{Bt^{\frac {n+p}{p}}e_p(u)^{\frac 1p}}=Ce^{g(t)}\, ,
\]
where $g(t)=Bt^{\frac {n+p}{p}}e_p(u)^{\frac 1p}$. For $s,t>0$ we have that
\begin{equation}\label{legendre transform}
\lambda(s)\leq \int_{\{z\in\Omega: u(z)<-s\}}e^{-st}e^{-tu}d\lambda_{2n}\leq e^{-st}\int_{\Omega}e^{-tu}d\lambda_{2n}\leq Ce^{-st+g(t)}.
\end{equation}
By Lemma~\ref{l1} we now have that
\[
g(t)-st=Bt^{\frac {n+p}{p}}e_p(u)^{\frac 1p}-st\geq -s^{\frac {n+p}{p}}\frac {np^{\frac pn}}{(n+p)^{1+\frac pn}}B^{-\frac pn}e_p(u)^{-\frac 1n}\, .
\]
Therefore, it follows from~(\ref{legendre transform}) that
\begin{equation}\label{lambda}
\lambda(s)\leq Ce^{-xs^{\frac {n+p}{p}}}, \, \text { where }\, x=\frac {np^{\frac pn}}{(n+p)^{1+\frac pn}}B^{-\frac pn}e_p(u)^{-\frac 1n}.
\end{equation}
By letting $r=xs^{\frac {n+p}{n}}$ in~(\ref{lambda}) we get that
\begin{multline}\label{2}
\|u\|_{L^q}^q=\int_{\Omega}(-u)^qd\lambda_{2n}=q\int_{0}^{\infty}s^{q-1}\lambda(s)ds\leq qC\int_0^{\infty}s^{q-1}e^{-xs^{\frac {n+p}{p}}}ds\\
=\frac {qnC}{(n+p)x^{\frac {nq}{n+p}}}\int_{0}^{\infty}r^{-1+\frac {nq}{n+p}}e^{-r}\,dr=\frac {qnC}{(n+p)x^{\frac {nq}{n+p}}}\, \Gamma \left (\frac {nq}{n+p}\right)\\
=Cx^{-\frac {nq}{n+p}}\, \Gamma \left (\frac {nq}{n+p}+1\right)=\frac {C(n+p)^{q}B^{\frac {pq}{n+p}}}{n^{\frac {nq}{n+p}}p^{\frac {pq}{n+p}}}\, \Gamma\left(\frac {nq}{n+p}+1\right)e_p(u)^{\frac {q}{n+p}}\, .
\end{multline}
Thus, (\ref{lq theorem}) holds for $u\in \Ep$. In the general case, let $u=u_1-u_2\in \delta\Ep$ it is enough to note that
\[
\|u\|_{L^q}\leq \|u_1+u_2\|_{L^q}\leq C(p,q,n,\Omega)e_p(u_1+u_2)\, ,
\]
and then taking the infimum over all possible decomposition of $u$.

Next we shall prove (\ref{q}). First assume that $u\in \Ep$. Note that for $y\geq 1$ it is a fact that
\[
\Gamma(y+1)\leq 2y^y\, ,
\]
so by (\ref{2}) it holds that for $q\geq 1$
\begin{multline*}
\|u\|_{L^q}\leq \left(\frac {C(n+p)^{q}B^{\frac {pq}{n+p}}}{n^{\frac {nq}{n+p}}p^{\frac {pq}{n+p}}}\Gamma\left(\frac {nq}{n+p}+1\right)e_p(u)^{\frac {q}{n+p}}\right)^{\frac 1q} \\
\leq 2^{\frac 1q}e^{\frac Aq}\frac {(n+p)B^{\frac {p}{n+p}}}{n^{\frac {n}{n+p}}p^{\frac {p}{n+p}}}\left(\frac {nq}{n+p}\right)^{\frac {n}{n+p}}e_p(u)^{\frac {1}{n+p}}\\
\leq 2e^{A}\frac {(n+p)^{\frac {p}{n+p}}B^{\frac {p}{n+p}}}{p^{\frac {p}{n+p}}}q^{\frac {n}{n+p}}e_p(u)^{\frac {1}{n+p}}\, .
\end{multline*}
To proceed to the general case $u=u_1-u_2\in \delta\Ep$ we follow the above procedure and arrive at
\[
\|u\|_{L^q}\leq D(p,n,\Omega)q^{\frac {n}{n+p}}\vertiii{u}_p\, ,
\]
with
\[
D(p,n,\Omega)=2e^{A(p,n,\Omega)}\frac {(n+p)^{\frac {p}{n+p}}B(p,n,\Omega)^{\frac {p}{n+p}}}{p^{\frac {p}{n+p}}}\, .
\]
To complete this proof we shall prove that the identity operator $\iota:\delta\Ep\to L^q$ is compact. Take a sequence $\{u_n\}=\{u_n^1-u_n^2\}\subset \delta\Ep$ with $\vertiii{u}_p\leq 1$. Then by the same reasoning as above we get that
\[
\|u_n^j\|_{L^q}\leq \|u_n^1+u_n^2\|_{L^q}\leq C(p,n,\Omega)\quad \text{for} \quad j=1,2\, .
\]
Hence, there exists a subsequence $\{u_{n_k}^j\}$ converging almost everywhere to some plurisubharmonic function $\{v^j\}$. This means that $\{u_{n_k}^1-u_{n_k}^2\}$ is a Cauchy sequence in $L^q$. Thus, $\iota$ is compact.
\end{proof}

The proof of Theorem~\ref{lq estimate} relies on the Moser-Trudinger inequality (Theorem~\ref{t11}). In the case when $q\leq n+p$,
we can present an elementary proof only using the inequalities in Lemma~\ref{est} and Theorem~\ref{thm_holder}.

\begin{proof}[Proof of Theorem~\ref{lq estimate} for $q\leq n+p$]
There exists $\varphi_0\in \Eo$ such that
\[
(dd^c\varphi_0)^n=d\lambda_{2n}\, ,
\]
(see e.g.~\cite{kol2}). Let $u=u_1-u_2\in \delta\Ep$, and let $0<q\leq p+n$. Thanks to Lemma~\ref{est} and Theorem~\ref{thm_holder} we get that
\begin{multline*}
\|u\|_{L^q}\leq \|u_1+u_2\|_{L^q}\leq \lambda_{2n}(\Omega)^{\frac 1q-\frac{1}{p+n}}\|u_1+u_2\|_{L^{p+n}}= \\
=\lambda_{2n}(\Omega)^{\frac 1q-\frac{1}{p+n}}\left(\int_{\Omega} (-u_1-u_2)^{p+n}(dd^c \varphi_0)^n\right)^{\frac 1{n+p}}\leq \\
\leq \lambda_{2n}(\Omega)^{\frac 1q-\frac{1}{p+n}}\left((p+n)\cdots (p+1)\|\varphi_0\|^n_{L^{\infty}}\int_{\Omega}(-u_1-u_2)^{p}(dd^c(u_1+u_2))^n\right)^{\frac {1}{n+p}}\leq \\
\leq \left(\lambda_{2n}(\Omega)^{\frac{p+n-q}{q}}(p+n)\cdots (p+1)\|\varphi_0\|_{L^{\infty}}^n\right)^{\frac {1}{n+p}}e_p(u_1+u_2)^{\frac {1}{n+p}}\, .
\end{multline*}
Finally by taking the infimum over all possible decompositions $u=u_1-u_2$ we obtain that
\[
\|u\|_{L^q}\leq C(p,q,n,\Omega)\vertiii{u}_p\, .
\]
\end{proof}

Next we present an example that shows that it is impossible to have an estimate of the type
\[
e_p(u)^{\frac {1}{n+p}}\leq C\|u\|_{L^q}\, .
\]

\begin{example}\label{ex1}
Consider the following functions defined on the unit ball $\mathbb{B}$ in $\C^n$
\[
u_j(z)=\frac 1j\max\left(\log |z|,-j^{1+\frac np}\right)\, .
\]
Then we have that
\[
u_j(z)=
\begin{cases}
\frac{1}{j} \log |z| & \text{ if }   \exp\left(-j^{1+\frac np}\right)\leq |z|\leq 1 \\[2mm]
-j^{\frac np} & \text{ if } 0\leq |z|\leq \exp\left(-j^{1+\frac np}\right)\, .
\end{cases}
\]
Hence, $\|u_j\|_{L^q}\to 0$, as $j\to \infty$, but at the same time we have that
\[
e_p(u)=\frac {1}{j^{n+p}}(2\pi)^n\left(j^{1+\frac np}\right)^p=(2\pi)^n\, ,
\]
which is a contradiction.
\hfill{$\Box$}
\end{example}

Example~\ref{ex2} shows that it is also impossible to have an estimate of the type
\[
\|u\|_{L^{\infty}}\leq C e_p(u)^{\frac {1}{n+p}}\, .
\]

\begin{example}\label{ex2}
Similarly as in Example~\ref{ex1} consider the following functions defined on the unit ball $\mathbb{B}$ in $\C^n$
\[
u_j(z)=\frac 1{j^{\frac {p}{n+p}}}\max\left(\log |z|,-j\right)\, .
\]
Then we have that $\|u_j\|_{L^{\infty}}=-u_j(0)=j^{\frac {n}{n+p}}\to \infty$, as $j\to \infty$, and at the same time
\[
e_p(u_j)=(2\pi)^nj^p\left(\frac 1{j^{\frac {p}{n+p}}}\right)^{n+p}=(2\pi)^n
\]
and a contradiction is obtained. \hfill{$\Box$}
\end{example}

Finally we present an example that shows that it is impossible to have an estimate of the type
\[
e_p(u)^{\frac {1}{n+p}}\leq C \|u\|_{L^{\infty}}\, .
\]

\begin{example}\label{ex3}
Similarly as before we consider the following functions defined on the unit ball $\mathbb{B}$ in $\C^n$
\[
u_j(z)=j\max\left(\log |z|,-\frac 1j\right)\, .
\]
Then we have that $\|u_j\|_{L^{\infty}}=-u_j(0)=1$ and at the same time
\[
e_p(u_j)=(2\pi)^nj^{n+p}\left(\frac 1j\right)^p=(2\pi)^nj^n\to \infty
\]
and a contradiction is obtained. \hfill{$\Box$}
\end{example}

Next in Corollary~\ref{t2_cor} we prove the corresponding Sobolev estimate~(\ref{lq estimate}) for functions in $\mathcal E_{\chi}$. For the definition of $\Ec$ see Definition~\ref{def_Echi} on page~\pageref{def_Echi}.

\begin{corollary}\label{t2_cor} Let $\chi:(-\infty,0]\to (-\infty,0]$ be a continuous and nondecreasing function, let $\Omega$ be a bounded hyperconvex domain in $\mathbb C^n$, $n\geq 2$, and let $u\in \Ec$. Then for all $q>0$ there exists a
constant $G(n,\Omega)\geq 0$ depending only on $n$ and $\Omega$ such that
\begin{multline*}
\|u\|_{L^q}\leq \\ G(n,\Omega)^{\frac 1q}\Gamma(q+1)^{\frac 1q}\left(\frac {e_{\chi}(u)}{(2\epsilon n)^n}\right)^{\frac 1n}\left(\inf_{s>0}s^{-q}\exp\left(-\left(\frac {e_{\chi}(u)}{(2\epsilon n)^n}\right)^{-\frac 1n}s\chi^{-1}(-s^n)\right)\right)^{\frac 1q}\, .
\end{multline*}
\end{corollary}
\begin{proof} This is a straight forward modification of the proof of Theorem~\ref{lq estimate}.
\end{proof}

We shall end this section with a remark about compact K\"{a}hler manifolds. This was first proved in~\cite{BB1} for the case $p=1$. The notation and background about the K\"ahler case are stated before the remark on page~\pageref{kahler}.

\begin{remark}\label{kahler2} Let $(X,\omega)$ be a compact K\"ahler manifold of dimension $n$ with a K\"ahler form $\omega$ such that $\int_{X}\omega^n=1$. Let $u\in \Ep$, $p>0$, and $k>0$. From (\ref{kahler 2}) we know that
\[
\log \left (\int_Xe^{-ku}\omega^n\right)\leq \frac {ak^{n+1}}{n!}e_p(u)^{\frac 1p}+b\, ,
\]
for some constants $a$ and $b$. By repeating the argument from the proof of Theorem~\ref{lq estimate} one can prove that there exists a constant $c$ depending only on $p, X$, and not on $q$, such that
\[
\|u\|_{L^q}\leq cq^{\frac {n}{n+p}}e_p(u)^{\frac {1}{n+p}}\, .
\]

\end{remark}

\section{On the Sobolev constant for the unit ball $\mathbb{B}$}\label{sec_optimal}

In this section let $C(p,q,n,\mathbb{B})$ be the infimum of all admissible constants in the Sobolev type inequality given
in~(\ref{lq theorem}). Our aim here is to show that
\[
C(1,1,n,\mathbb{B})=\frac {\pi^{\frac {n^2}{n+1}}}{4^{\frac {n}{n+1}}n!(n+1)^{\frac {n}{n+1}}}\, .
\]
We shall do it in two part as follows.
\begin{enumerate}

\item In Example~\ref{const1}, we derive that for $q\leq n+1$
\[
C(1,q,n,\mathbb{B})\leq \frac {\pi^{\frac {n(1+n-q)}{q(n+1)}}[1\cdots(1+\lceil q-1 \rceil)]^\frac {1}{1+\lceil q-1 \rceil}}{4^{\frac {n}{n+1}}(n!)^{\frac 1q}(n+1)^{\frac {n-\lceil q-1 \rceil}{(n+1)(1+\lceil q-1 \rceil)}}}\, ,
\]
where $\lceil \,\cdot\, \rceil$ is the ceiling function.

\item In Example~\ref{const2} we prove that
\[
C(p,1,n,\mathbb{B})\geq\frac {\pi^{\frac {n(n+p-1)}{n+p}}p^{\frac {p}{n+p}}}{4^{\frac n{n+p}}n!(n+p)(n\mathbf {B}(p+1,n))^{\frac {1}{n+p}}}\, ,
\]
where $\mathbf B$ is the beta function. We shall actually obtain a bit more general result in this example.
\end{enumerate}

\bigskip

\begin{example}\label{const1}
Let $\Omega=\mathbb{B}$ be the unit ball in $\mathbb C^n$, and on $\mathbb{B}$ define
\[
\varphi_0=\frac {1}{4\root n \of {n!}}(|z|^2-1)\, .
\]
Then $\varphi_0\in\Eo$, $(dd^c\varphi_0)^n=d\lambda_{2n}$, and
\[
e_p(\varphi_0)=\frac {n\pi^n}{4^p(n!)^{1+\frac pn}}\mathbf B(p+1,n)\, ,
\]
where $\mathbf B$ is the classical beta Euler function. Recall that if $q\leq n+p$, then the ceiling function evaluated at $q-p$ is defined by
\[
\lceil q-p \rceil =\min\{k\in \mathbb N: k\geq q-p\}\, .
\]
Once again thanks to Lemma~\ref{est} and Theorem~\ref{thm_holder} we get that
\begin{multline}\label{ineq2}
\|u\|_{L^q}\leq \lambda_{2n}(\mathbb{B})^{\frac 1q-\frac {1}{\lceil q-p \rceil+p}}\|u\|_{L^{p+\lceil q-p \rceil}}\leq \\
\leq  \lambda_{2n}(\mathbb{B})^{\frac 1q-\frac {1}{\lceil q-p \rceil+p}}\left((p+1)\cdots (p+\lceil q-p \rceil)\|\varphi_0\|^{\lceil q-p \rceil}_{L^{\infty}}\right. \\
\cdot\left.\int_{\Omega}(-u)^{p}(dd^cu)^{\lceil q-p \rceil}\wedge(dd^c\varphi_0)^{n-\lceil q-p \rceil}\right)^{\frac 1{p+\lceil q-p \rceil}}\\
\leq \lambda_{2n}(\mathbb{B})^{\frac 1q-\frac {1}{p+\lceil q-p \rceil}}\left((p+1)\cdots(p+\lceil q-p \rceil)\|\varphi_0\|^{\lceil q-p \rceil}_{L^{\infty}}d(p,n,\mathbb{B})\right.\\ \cdot \left. e_p(\varphi_0)^{\frac {n-\lceil q-p \rceil}{n+p}}e_p(u)^{\frac {p+\lceil q-p \rceil}{n+p}}\right)^{{\frac 1{p+\lceil q-p \rceil}}}=\\
=\frac{e_p(u)^{\frac {1}{n+p}}}{{4^{\frac {n}{n+p}}(n!)^{\frac 1q}}}\pi^{\frac {n(p+n-q)}{q(n+p)}}d(p,n,\mathbb{B})^{\frac {1}{p+\lceil q-p \rceil}}\Big(n \mathbf B(p+1,n)\Big)^{\frac {n-\lceil q-p \rceil}{(n+p)(p+\lceil q-p \rceil)}} \\
\cdot\left((p+1)\cdots(p+\lceil q-p \rceil)\right)^\frac {1}{p+\lceil q-p \rceil}\, .
\end{multline}

If $p=1$ we know that $d(1,n,\mathbb{B})=1$, and $n\mathbf B(2,n)=\frac {1}{n+1}$. Hence,
\begin{equation}\label{ineq5}
C(1,q,n,\mathbb{B})\leq \frac {\pi^{\frac {n(1+n-q)}{q(n+1)}}[(1+\lceil q-1 \rceil)!]^\frac {1}{1+\lceil q-1 \rceil}}{4^{\frac {n}{n+1}}(n!)^{\frac 1q}(n+1)^{\frac {n-\lceil q-1 \rceil}{(n+1)(1+\lceil q-1 \rceil)}}}.
\end{equation}
\hfill{$\Box$}
\end{example}

\begin{example}\label{const2}
For $\alpha>0$, $k>0$, define on the unit ball $\mathbb{B}$ in $\mathbb C^n$ the following family of functions
\[
u_{\alpha, k}(z)=k(|z|^{2\alpha}-1)\, .
\]
Then we have that
\[
e_p(u_{\alpha})=\int_{B(0,1)}(-u_{\alpha})^p\ddcn{u_{\alpha}}=k^{n+p}n(4\pi)^n\alpha^n\, \mathbf {B}(p+1,n)\, ,
\]
and
\begin{multline*}
\int_{\mathbb{B}}(-u_{\alpha,k}(z))^qd\lambda_{2n}=\frac {2\pi^nk^q}{(n-1)!}\int_{0}^1(1-t^{2\alpha})^qt^{2n-1}dt\\
=\frac {\pi^nk^q}{(n-1)!\alpha}\int_0^1(1-s)^qs^{\frac {n}{\alpha}-1}dr=\frac {\pi^nk^q}{(n-1)!\alpha}\mathbf B\left(q+1,\frac {n}{\alpha}\right).
\end{multline*}
Hence,
\[
C(p,q,n,\mathbb{B})\geq\frac {\|u_{\alpha,k}\|_{L^q}}{e_p(u_{\alpha,k})^{\frac {1}{n+p}}}=\frac {\pi^{\frac {n(n+p-q)}{q(n+p)}}n^{\frac {n+p-q}{q(n+p)}}}{4^{\frac n{n+p}}(n!)^{\frac 1q}\mathbf B(p+1,n)^{\frac {1}{n+p}}}\frac {\mathbf B\left(q+1,\frac {n}{\alpha}\right)^{\frac 1q}}{\alpha^{\frac {1}{q}+\frac {n}{n+p}}}\, .
\]
Now set $\beta=\frac {n}{\alpha}$, and $s=\frac {n}{n+p}$. With these notation we get that
\[
\frac {\mathbf B(q+1,\frac {n}{\alpha})^{\frac 1q}}{\alpha^{\frac {1}{q}+\frac {n}{n+p}}}=n^{-\frac {1}{q}-\frac {n}{n+p}}\left(\mathbf B(q+1,\beta)\beta^{1+q\frac {n}{n+p}}\right)^{\frac 1q}=n^{-\frac {1}{q}-\frac {n}{n+p}}f(\beta)^{\frac 1q}\, ,
\]
where
\[
f(\beta)=\mathbf B(q+1,\beta)\beta^{1+qs}\, .
\]
Next we want to find $\sup_{(0,\infty)}f(\beta)$. First we see that
\[
\lim_{\beta\to 0}f(\beta)=\lim_{\beta\to \infty}f(\beta)=0\, ,
\]
and that
\[
f'(\beta)=\mathbf B(q+1,\beta)\beta^{sq}\left(\beta(\psi(\beta)-\psi(\beta+q+1))+1+sq\right)\, ,
\]
where $\psi(x)$ is the classical digamma function.

In the case when $q\in \mathbb N$, it holds that
\[
\psi(\beta)-\psi(\beta+q+1)=-\sum_{j=0}^q\frac {1}{\beta+q-j}\, .
\]
Therefore for $q\in \mathbb N$ we have that $f'(\beta)=0$ if, and only if,
\begin{equation}\label{sol}
\sum_{j=0}^q\frac {\beta}{\beta+q-j}=1+qs\, .
\end{equation}
This implies that in the case $q=1$ the equation (\ref{sol}) have a solution given by
\[
\beta_0=\frac {s}{1-s}=\frac np\, .
\]
Using the standard equalities
\[
\mathbf B(x+1,y)=\frac {x}{x+y}\mathbf B(x,y)\, ,  \quad \text{ and } \quad \mathbf B(1,y)=\frac 1y\, ,
\]
we derive that
\[
\sup f(\beta)=\mathbf B(q+1,\beta_0)\beta_0^{1+s}=\frac {\beta_0^{s}}{1+\beta_0}=\frac {p^{\frac {p}{n+p}}n^{\frac {n}{n+p}}}{n+p}\, .
\]
Hence,
\begin{equation}\label{ineq6}
C(p,1,n,\mathbb{B})\geq\frac {\pi^{\frac {n(n+p-1)}{n+p}}p^{\frac {p}{n+p}}}{4^{\frac n{n+p}}n!(n+p)(n\mathbf B(p+1,n))^{\frac {1}{n+p}}}.
\end{equation}

In the case when $q\in \mathbb N$, $q\geq 2$, then we have that
\[
\beta_0=\frac {q+1}{2}\frac {s}{1-s}=\frac {(q+1)n}{2p}
\]
is a good approximation of a solution to the equation (\ref{sol}), and therefore
\[
\sup f(\beta)\geq \mathbf B\left(q+1,\frac {(q+1)n}{2p}\right)\left(\frac {(q+1)n}{2p}\right)^{1+\frac {n}{n+p}}\, .
\]
Thus,
\begin{equation}\label{ineq7}
C(p,q,n,\mathbb{B})\geq\frac {\pi^{\frac {n(n+p-q)}{q(n+p)}}n^{\frac {2n+p-nq}{q(n+p)}}(q+1)^{\frac {2n+p}{q(n+p)}}}{4^{\frac n{n+p}}(2p)^{\frac {2n+p}{q(n+p)}}(n!)^{\frac 1q}(nB(p+1,n))^{\frac {1}{n+p}}}\mathbf B\left(q+1,\frac {(q+1)n}{2p}\right)^{\frac 1q}\, .
\end{equation}
For $q\in\RE$, $q\geq 2$, one can insert the floor function evaluated at $q$, $\lfloor q\rfloor$, in~(\ref{ineq7}).
\hfill{$\Box$}
\end{example}

\end{document}